\documentclass[10pt]{article}
\usepackage{graphicx}
\usepackage[utf8]{inputenc}
\usepackage{amsmath}
\usepackage{combelow}
\usepackage[pdftex,pdfpagelabels]{hyperref}
\usepackage{combelow}
\usepackage{amssymb}
\usepackage{amsthm}
\usepackage{enumerate}
\usepackage{fullpage}
\usepackage{url}
\usepackage{float}
\usepackage{verbatim}
\usepackage{xcolor}
\usepackage{afterpage}

\usepackage[style=numeric, giveninits=true, maxnames=4]{biblatex}
\theoremstyle{plain}
\newtheorem{thm}{Theorem}[section]
\newtheorem{defi}[thm]{Definition}
\newtheorem{lem}[thm]{Lemma}

\newtheorem{prop}[thm]{Proposition}
\newtheorem{rmk}[thm]{Remark}

\newcommand{\R}{\mathbb{R}}
\newcommand{\N}{\mathbb{N}}

\DeclareMathOperator{\Id}{Id}

\DeclareMathOperator{\trace}{tr}
\DeclareMathOperator{\argmin}{argmin}

\allowdisplaybreaks

\title{Asymptotic behaviour of stochastic inertial dynamics incorporating a Tikhonov regularization term}
\author{Chiara Schindler \footnote{Faculty of Mathematics, University of Vienna, Oskar-Morgenstern-Platz 1, 1090 Vienna, Austria, e-mail: \url{chiara.schindler@univie.ac.at}.}}

\begin{document}

\maketitle

\begin{abstract}
    In a separable Hilbert space, we study the minimization problem of a convex smooth function with Lipschitz continuous gradient whose evaluations are corrupted by random noise. To this end, we associate a stochastic inertial system that incorporates Tikhonov regularization with the optimization problem. We establish existence and uniqueness of a solution trajectory for this system. Then, we derive an upper bound on the expected value of an appropriate associated energy function given square-integrability of the diffusion $\sigma_X$ before focusing on the particular case where the parameter function multiplied by the Tikhonov term is given by $\frac{1}{t^r}$ for $0<r<2$. For this setting, we show a.s. convergence rates as well as convergence rates in expectation for the function values along the trajectory to an infimal value, the trajectory process to an optimal solution and its time derivative to zero under a stronger integrability condition on $\sigma_X$.
\end{abstract}


\section{Introduction}

Consider the minimization problem of a convex differentiable function $f:\mathcal{H}\rightarrow \R$ with Lipschitz continuous gradient, where $\mathcal{H}$ denotes a separable real Hilbert space,
\begin{align*}
    \inf_{x\in\mathcal{H}} f(x),
\end{align*}
and assume that its set of solutions $\{x^* \in \mathcal{H}: f(x^*) = \inf_{x\in\mathcal{H}} f\}$ is nonempty, where $\inf_{x\in\mathcal{H}} f\in\R$ denotes the infimal value of $f$. We attach the following Tikhonov regularized inertial system to the minimization problem:
\begin{align}\label{eq:s-TRIGS}\tag{S-TRIGS}
\begin{cases}
    dX(t) = Y(t)dt
    \\ dY(t) = (-\delta \sqrt{\varepsilon(t)}Y(t) - \nabla f(X(t)) - \varepsilon(t) X(t))dt + \sigma_X(t) dW_X(t)
    \\ X(t_0) = X_0, \ Y(t_0) = Y_0,
\end{cases}
\end{align}
defined on a filtered probability space $(\Omega,\mathcal{F},\{\mathcal{F}_t\}_{t\geq t_0},\mathbb{P})$, where $W$ denotes a $\mathcal{H}$-valued Brownian motion, the diffusion term, denoted by $\sigma_X :[t_0,+\infty)\rightarrow \mathcal{L}(\mathcal{H},\mathcal{H})$, is a measurable map, $\varepsilon$ is a nonnegative function such that $\lim_{t\rightarrow +\infty} \varepsilon(t) = 0$ and $\delta>0, \ \lambda>0$. Note that $\varepsilon$ is present in both the viscous damping and the Tikhonov regularization terms, so $\lim_{t\rightarrow +\infty} \varepsilon(t) = 0$ leads to asymptotic vanishing of the damping coefficient in coordination with the Tikhonov regularization coefficient.

We aim to establish existence and uniqueness of trajectory solutions of~\eqref{eq:s-TRIGS} and to investigate its asymptotic behaviour. First, we derive some convergence results for $f(X(t)) - \inf_{\mathcal{H}} f$ and $\|X(t)-x_{\varepsilon(t)}\|$, where $x_{\varepsilon(t)} := \argmin \left\{ f+ \frac{\varepsilon(t)}{2}\|\cdot\|^2\right\}$ tends to the minimum norm minimizer of $f$ as $t\rightarrow +\infty$, for general $\varepsilon$ satisfying the conditions laid out above as well as~\eqref{eq:condition-tikhonov-parameter}, a technical requirement on the Tikhonov parameter that is necessary for the proofs of this paper. As a next step, we move to the particular case $\varepsilon(t) = \frac{1}{t^r}$, where $0<r<2$, and derive the following convergence rates a.s. as well as in expectation: We show convergence rates for $f(X(t)) - \inf_{\mathcal{H}} f$, $\|X(t)-x_{\varepsilon(t)}\|^2$ and $\|Y(t)\|$ as $t\rightarrow +\infty$ under the condition that $\int_{t_1}^{+\infty} \gamma(s)\|\sigma_X(s)\|_{HS}^2 ds < +\infty$, where $\gamma(t) = \exp\left(\int_{t_1}^{t} - \frac{\dot{\varepsilon}(s)}{2\varepsilon(s)} + (\delta - \lambda)\sqrt{\varepsilon(s)} ds\right) = \left(\frac{t}{t_1}\right)^\frac{r}{2}\exp\left(\frac{2(\delta-\lambda)}{2-r}\left(t^{1-\frac{r}{2}} - t_1^{1-\frac{r}{2}}\right)\right)$, where $t_1>t_0$ is introduced by the condition~\eqref{eq:condition-tikhonov-parameter}. The integral condition would be the satisfied e.g. for $\sigma(t) := \exp(-t^{1-\frac{r}{2}+\alpha})\Id$, where $\alpha>0$.

Note that the system~\eqref{eq:s-TRIGS} is indeed the stochastic version of a Tikhonov regularized inertial gradient system
\begin{align}\label{eq:TRIGS}\tag{TRIGS}
    \ddot{x}(t) + \delta\sqrt{\varepsilon(t)}\dot{x}(t) + \nabla f(x(t)) + \varepsilon(t)x(t) = 0,
\end{align}
which was introduced in~\cite{attouch-laszlo} and considered more generally with an added Hessian driven damping term, so
\begin{align*}
    \ddot{x}(t) + \delta\sqrt{\varepsilon(t)}\dot{x}(t) + \beta\nabla^2 f(x(t))\dot{x}(t) + \nabla f(x(t)) + \varepsilon(t)x(t) = 0,
\end{align*}
in~\cite{attouch-balhag-chbani-riahi}, where it is shown to exhibit the same convergence rates that we prove in the stochastic setting in the sections~\ref{sec:general-epsilon} and~\ref{sec:particular-case-epsilon-1/t^r}.

Consider the Tikhonov regularized system with asymptotic vanishing damping
\begin{align}\label{eq:avd-tikh}
    \ddot{x}(t) + \frac{\alpha}{t}\dot{x}(t) + \nabla f(x(t)) + \varepsilon(t)x(t) = 0,
\end{align}
studied in~\cite{attouch-chbani-riahi-fast-inertial-dynamics-with-tikhonov}. This is related to the non-Tikhonov regularized system via strong convexity: Replace the convex smooth function by the strongly convex smooth mapping $x\mapsto f(x) + \frac{\varepsilon(t)}{2}\|x\|^2$ to obtain~\eqref{eq:avd-tikh} from the non-Tikhonov regularized version
\begin{align*}
    \ddot{x}(t) + \frac{\alpha}{t}\dot{x}(t) + \nabla f(x(t)) = 0,
\end{align*}
which was introduced by Su, Boyd and Cand\`{e}s in~\cite{su-boyd-candes} and exhibits a rate of convergence $f(x(s)) - \inf_{\cal H} = \mathcal{O}\left(\frac{1}{s^2}\right)$ as $s \to +\infty$ given that $\alpha \geq 3$. Thus, it outperforms the classical gradient flow, for which we have the rate of convergence $f(x(s)) - \inf_{\cal H} = o\left(\frac{1}{s}\right)$ as $s \to +\infty$. Even more, according to~\cite{attouch-chbani-peypoquet-riahi}, if $\alpha >3$, then $f(x(s)) - \inf_{\cal H} = o\left(\frac{1}{s^2}\right)$, and the trajectory solution $x(s)$ converges weakly to a minimizer of $f$ as $s \to +\infty$. An explicit discretization of this system leads to the Nesterov accelerated gradient method~\cite{nesterov-acc-gradient, chambolle-dossal}, which shares the asymptotic properties of the continuous time dynamics. The stochastic version of this, given by
\begin{align}\label{eq:SAVD-alpha}\tag{SAVD}
    \begin{cases}
        dX(t) = Y(t)dt
        \\ dY(t) = \left(-\frac{\alpha}{t}Y(t) - \nabla f(X(t))\right) dt + \sigma_{X}(t) dW(t)
        \\ X(t_0) = X_0, \ \mbox{} \ Y(t_0) = Y_0,
    \end{cases}
\end{align}
was examined in~\cite{mfao-stochasticintertialgradient} and arose again via the methodology of time rescaling in~\cite{bot-s-heavy-ball}.

Following a similar line of reasoning, the system~\eqref{eq:TRIGS} can be derived from the following equation describing the heavy ball:
\begin{align}\label{eq:heavy-ball}
    \ddot{x}(t) + 2\sqrt{\mu}\dot{x}(t) + \nabla f(x(t)) = 0,
\end{align}
which was proved in~\cite{polyak-speedingupconvergence} to enjoy the convergence property $f(x(t)) - \inf_\mathcal{H} f = \mathcal{O}\left(\exp\left(-\sqrt{\mu}t\right)\right)$ as $t\rightarrow +\infty$ if $f$ is $\mu$-strongly convex. Again replacing the convex smooth function $f$ by the $\varepsilon(t)$-strongly convex smooth mapping $x\mapsto f(x) + \frac{\varepsilon(t)}{2}\|x\|^2$ yields~\eqref{eq:TRIGS} for~\eqref{eq:heavy-ball}. 

It is worth noting that the accelerated convergence of function values observed in the deterministic Su–Boyd–Cand\`{e}s dynamical system, when compared to the classical gradient flow, also manifests in the stochastic setting. In particular, analogous improvements arise for the stochastic differential equation~\eqref{eq:SAVD-alpha} when compared with the stochastic gradient flow
\begin{align*}
    \begin{cases}
        dX(t) = -\nabla f(X(t))dt + \sigma_X(t)dW(t)
        \\ X(t_0) = X_0,
    \end{cases}
\end{align*}
that has been explored in~\cite{soto-fadili-attouch}.

A Tikhonov regularized differential inclusion without momentum term
\begin{align*}
\begin{cases}
    dX(t)\in - \partial F(X(t)) - \varepsilon(t)X(t) + \sigma(t,X(t)) dW(t),
    \\ X(t_0) = X_0
\end{cases}
\end{align*}
has been examined in~\cite{mfa-tikh}: For this system, almost sure strong convergence of the trajectory was achieved under suitable assumptions on the parameter function $\varepsilon$. Given different requirements, we prove an analogous result for the inertial Tikhonov regularized system in the present paper. Furthermore, the inertial setting additionally allows for an almost sure bound on the function values.

\section{Results for a general Tikhonov regularization parameter}\label{sec:general-epsilon}
Let $\mathcal{H}$ be a real separable Hilbert space. Consider the second-order dynamic
\begin{align*}
    \ddot{x}(t) + \delta\sqrt{\varepsilon(t)}\dot{x}(t) + \nabla f(x(t)) + \varepsilon(t)x(t) = 0 \ \mbox{on} \ [t_0,\infty) \ \mbox{with} \ t_0>0
\end{align*}
introduced in~\cite{attouch-balhag-chbani-riahi} for the deterministic setting. Throughout the following, assume that $f$ and $\varepsilon$ satisfy
\begin{align}\tag{A}\label{eq:assumption-f}
    \begin{cases}
        f:\mathcal{H}\rightarrow \R \ \mbox{is convex, of class} \ \mathcal{C}^1, \ \mbox{with} \ L \mbox{-Lipschitz continuous gradient;}
        \\ S:=\argmin_\mathcal{H} f \neq \emptyset. \ \mbox{We denote the element of minimum norm of} \ S \ \mbox{by} \ x^*;
        \\ \varepsilon:[t_0,+\infty)\rightarrow \R^+ \ \mbox{is a nonincreasing function, of class} \ \mathcal{C}^1, \ \mbox{such that} \ \lim_{t\rightarrow + \infty} \varepsilon(t) = 0.
    \end{cases}
\end{align}
In order to examine this system in the stochastic setting, consider stochastic It\^{o} processes $X,Y$ satisfying
\begin{align}\tag{S-TRIGS}
\begin{cases}
    dX(t) = Y(t)dt
    \\ dY(t) = (-\delta \sqrt{\varepsilon(t)}Y(t) - \nabla f(X(t)) - \varepsilon(t) X(t))dt + \sigma_X(t) dW_X(t),
\end{cases}
\end{align}
where we require $\sigma_X(t)$ to be square-integrable, i.e. $\int_{t_1}^{+\infty} \|\sigma_X(s)\|^2ds < +\infty$.

In our analysis, we shall further assume that $\varepsilon(\cdot)$ satisfies the growth condition detailed below.
\begin{defi}
    The Tikhonov regularization parameter $\varepsilon(\cdot)$ satisfies the condition~\eqref{eq:condition-tikhonov-parameter} if there exists $a>1$, $c>2$, $\lambda>0$ and $t_1\geq t_0$ such that for all $t\geq t_1$:
    \begin{align}\label{eq:condition-tikhonov-parameter}
        \frac{d}{dt}\left(\frac{1}{\sqrt{\varepsilon(t)}}\right)\leq \min \left(2\lambda-\delta, \frac{1}{2}\left(\delta-\frac{a+1}{a}\lambda\right)\right),
    \end{align}
    where $\frac{\delta}{2}<\lambda<\delta$ and
    \begin{align*}
        \frac{1}{2}\left(\delta+\frac{1}{c} + \sqrt{\left(\delta+\frac{1}{c}\right)^2 - 2}\right) < \lambda < \min\left(\frac{a}{a+1}\delta, \frac{\delta+\sqrt{\delta^2-4}}{2}\right) \ \mbox{when} \ \delta>2.
    \end{align*}
\end{defi}
Before we can move to convergence results, we have to establish the existence and uniqueness of the solution of~\eqref{eq:s-TRIGS}.
\begin{thm}
    If $\sigma_X:\R\rightarrow L^2(\mathcal{H},\mathcal{H})$ is square-integrable, then the SDE~\eqref{eq:s-TRIGS} has a unique solution $\in S_\mathcal{H}^\nu$, where $\nu\geq 2$. Further, this solution is a strong solution.
\end{thm}
\begin{proof}
    For convenience, we recall~\eqref{eq:s-TRIGS}:
    \begin{align*}
        \begin{cases}
            dX(t) = Y(t)
            \\ dY(t) = (-\delta\sqrt{\varepsilon(t)}Y(t) - \nabla f(X(t)) - \varepsilon(t)X(t))dt + \sigma_X(t)dW_X(t).
        \end{cases}
    \end{align*}
    Next, introduce the notation
    \begin{align*}
        \sigma_{X,*} := \sup_{t\in[t_0,\infty)} \|\sigma_X(t)\|_{HS},
    \end{align*}
    where $\|\cdot\|_{HS}$ denotes the Hilbert-Schmidt norm. Now fix $T>t_0$ and consider the system for $t\in[t_0,T]$. It holds
    \begin{align*}
        \|y_1 &- y_2\|^2 + \|-\delta\sqrt{\varepsilon(t)}(y_1 - y_2) - \nabla f(x_1) + \nabla f(x_2) - \varepsilon(t)(x_1-x_2)\|^2
        \\&\leq \|y_1-y_2\|^2 + 3\delta^2\varepsilon(t)\|y_1-y_2\|^2 + 3(\varepsilon(t)^2 + L^2)\|x_1-x_2\|^2
        \\&\leq \max(3\delta^2\varepsilon(t), 3(\varepsilon(t)^2 + L^2))(\|y_1-y_2\|^2 + \|x_1-x_2\|^2)
        \\&\leq \sup_{t\in[t_0,T]}\left(\max\left(3\delta^2\varepsilon(t), 3(\varepsilon(t)^2 + L^2)\right)\right)\left\|\begin{pmatrix}
        x_1-x_2
        \\ y_1-y_2\end{pmatrix}\right\|^2,
    \end{align*}
    so Theorem~\ref{thm:existence-uniqueness-solutions}~\eqref{thm:item:help-existence-uniqueness} yields the fact that the solution lies in $S_\mathcal{H}^\nu$ for $\nu\geq 2$. In order to prove that it is a strong solution, verify
    \begin{align*}
        &\left\|\begin{pmatrix}
            Y(t)
            \\ -\delta\sqrt{\varepsilon(t)}Y(t) - \nabla f(X(t)) - \varepsilon(t)X(t)
        \end{pmatrix}\right\| + \|\sigma_X(t)\|_{HS}
        \\&\quad\leq \sqrt{\|Y(t)\|^2 + \|-\delta\sqrt{\varepsilon(t)}Y(t) - \nabla f(X(t))-\varepsilon(t)X(t)\|^2} + \sigma_{X,*}
        \\&\quad\leq \sqrt{\|Y(t)\|^2 + 3\delta^2\varepsilon(t)\|Y(t)\|^2 + 3\|\nabla f(X(t))\|^2 + 3\varepsilon(t)^2\|X(t)\|^2} + \sigma_{X,*}
        \\&\quad\leq \sqrt{(1+3\delta^2\varepsilon(t))\|Y(t)\|^2 + 3\|\nabla f(X(t))\|^2 + 3\varepsilon(t)^2\|X(t)\|^2} + \sigma_{X,*}
        \\&\quad\leq \sqrt{(1+3\delta^2\varepsilon(t))\|Y(t)\|^2 + (6L^2+3\varepsilon(t)^2)\|X(t)\|^2 + 6L^2\|x^*\|^2} + \sigma_{X,*}
        \\&\quad\leq \sqrt{\sup_{t\in[t_0,T]}\left(\max\left((1+3\delta^2\varepsilon(t)), (6L^2+3\varepsilon(t)^2), 6L^2\|x^*\|^2, \sigma_{X,*}^2\right)\right)} \left(1+\sup_{t\in[t_0,T]}\left\|\begin{pmatrix}
            X(t)
            \\Y(t)
        \end{pmatrix}\right\|\right),
    \end{align*}
    thus satisfying the necessary conditions for Theorem~\ref{thm:existence-uniqueness-solutions}~\eqref{thm:item:existence-uniqueness-strong-solution} and proving the desired claim.
\end{proof}

In order to define the energy function we will be working with in the following theorem, introduce
\begin{align*}
    \varphi_t(x) := f(x) + \frac{\varepsilon(t)}{2}\|x\|^2, \hfill x_{\varepsilon(t)} := \argmin\left\{ f+\frac{\varepsilon(t)}{2}\|\cdot\|^2\right\}, \hfill V(t,x,y):= \lambda\sqrt{\varepsilon(t)}(x-x_{\varepsilon(t)}) + y.
\end{align*}
Now set
\begin{align}\label{eq:def-energy-function}
    \mathcal{E}(t,x,y) := \varphi_t(x)-\varphi_t(x_{\varepsilon(t)}) + \frac{1}{2}\|V(t,x,y)\|^2.
\end{align}
\begin{thm}\label{thm:rates-in-expectation}
    Let $\sigma_X$ be square-integrable and denote the minimum norm minimizer of $f$ by $x^*$. Then, it holds
    \begin{align*}
        \mathbb{E}(\mathcal{E}(&t,X(t),Y(t)))
        \\&\leq \frac{\gamma(t_1)}{\gamma(t)}\mathbb{E}(\mathcal{E}(t_1,X(t_1),Y(t_1))) + \frac{1}{\gamma(t)}\int_{t_1}^t\frac{\|x^*\|^2}{2}G(s)\gamma(s) ds + \int_{t_1}^{+\infty}\|\sigma_X(s)\|_{HS}^2ds \ \forall t\geq t_1,
\end{align*}
where
\begin{align*}
    G(t):= (a+c)\lambda\frac{\dot{\varepsilon}(t)^2}{\varepsilon(t)^\frac{3}{2}} - \dot{\varepsilon}(t), \ \mu(t) = -\frac{\dot{\varepsilon}(t)}{2\varepsilon(t)} + (\delta-\lambda)\sqrt{\varepsilon(t)}, \ \gamma(t) = \exp\left(\int_{t_1}^t \mu(s)ds\right).
\end{align*}
If the stronger condition $\int_{t_1}^{+\infty}\gamma(s)\|\sigma_X(s)\|_{HS}^2ds < +\infty$ holds, we even have
\begin{align*}
    \mathbb{E}(\mathcal{E}(&t,X(t),Y(t)))
    \\&\leq \frac{\gamma(t_1)}{\gamma(t)}\mathbb{E}(\mathcal{E}(t_1,X(t_1),Y(t_1))) + \frac{1}{\gamma(t)}\int_{t_1}^t\frac{\|x^*\|^2}{2}G(s)\gamma(s) ds + \frac{1}{\gamma(t)}\int_{t_1}^{+\infty}\gamma(s)\|\sigma_X(s)\|_{HS}^2ds \ \forall t\geq t_1.
\end{align*}
\end{thm}
\begin{proof}
First, observe
\begin{align*}
    &\frac{d}{dt}\mathcal{E}(t,x,y) = \frac{\dot{\varepsilon}(t)}{2}\|x\|^2 - \frac{\dot{\varepsilon}(t)}{2}\|x_{\varepsilon(t)}\|^2 + \left\langle V(t,x,y), \lambda\frac{\dot{\varepsilon}(t)}{2\sqrt{\varepsilon(t)}}(x-x_{\varepsilon(t)}) - \lambda\sqrt{\varepsilon(t)}\frac{d}{dt}x_{\varepsilon(t)}\right\rangle
    \\& \frac{d}{dx}\mathcal{E}(t,x,y) = \nabla f(x) + \varepsilon(t)x + \lambda^2\varepsilon(t)(x-x_{\varepsilon(t)}) + \lambda\sqrt{\varepsilon(t)}y
    \\&\frac{d}{dy}\mathcal{E}(t,x,y) = \lambda\sqrt{\varepsilon(t)}(x-x_{\varepsilon(t)}) + y
    \\&\frac{d^2}{dy^2}\mathcal{E}(t,x,y) = I.
\end{align*}
Note that we can apply It\^{o}'s formula: The condition ``bounded on bounded subsets of $[t_0,+\infty)\times\mathcal{H}\times\mathcal{H}$" holds since, $x^*$ being a minimizer of $f$, the Lipschitz continuity of $\nabla f$ yields
\begin{align*}
    \|\nabla f(x)\| = \|\nabla f(x) - \nabla f(x^*)\| \leq L\|x-x^*\|,
\end{align*}
which is bounded on bounded subsets of $\mathcal{H}$. Furthermore,
\begin{align*}
    \left\|\frac{d}{dt}x_{\varepsilon(t)}\right\| \leq -\frac{\dot{\varepsilon}(t)}{\varepsilon(t)}\|x_{\varepsilon(t)}\| \leq-\frac{\dot{\varepsilon}(t)}{\varepsilon(t)}\|x^*\|
\end{align*}
is also bounded. Concerning continuity, the only aspect that is not immediate is continuity of $t\mapsto \frac{d}{dt}x_{\varepsilon(t)}$. However, applying the implicit function theorem for $F(t,x) = \nabla f(x) + \varepsilon(t)x$ yields the existence of a continuously differentiable $x_{\varepsilon(t)}$ such that $F(t,x_{\varepsilon(t)}) = 0$ and $F(t,x) = 0 \Leftrightarrow x = x_{\varepsilon(t)}$ (in order to obtain $x_{\varepsilon(t)}$ on the entire set of $[t_0,+\infty)$, observe that the conditions for the implicit function theorem are satisfied for arbitrary $t\in[t_0,+\infty)$ and $x\in\mathcal{H}$, apply the implicit function theorem repeatedly and use uniqueness of $x_{\varepsilon(t)}$ whenever the environments overlap to show that $x_{\varepsilon(t)}$ can be extended to the whole of $[t_0,+\infty)$).

Therefore, by It\^{o}'s formula,
\begin{align*}
    d\mathcal{E}(&t,X(t),Y(t)) =
    \\&\Bigg(\frac{\dot{\varepsilon}(t)}{2}\|X(t)\|^2 - \frac{\dot{\varepsilon}(t)}{2}\|x_{\varepsilon(t)}\|^2 + \left\langle V(t,X(t),Y(t)), \lambda\frac{\dot{\varepsilon}(t)}{2\sqrt{\varepsilon(t)}}(X(t)-x_{\varepsilon(t)}) - \lambda\sqrt{\varepsilon(t)}\frac{d}{dt}x_{\varepsilon(t)}\right\rangle
    \\&\qquad + \langle \nabla f(X(t)) + \varepsilon(t)X(t) + \lambda^2\varepsilon(t)(X(t)-x_{\varepsilon(t)}) + \lambda\sqrt{\varepsilon(t)}Y(t), Y(t)\rangle
    \\&\qquad+ \langle \lambda\sqrt{\varepsilon(t)}(X(t)-x_{\varepsilon(t)}) + Y(t), -\delta \sqrt{\varepsilon(t)}Y(t) - \nabla f(X(t)) - \varepsilon(t)X(t)\rangle\Bigg) dt
    \\&\quad + \langle \lambda\sqrt{\varepsilon(t)}(X(t)-x_{\varepsilon(t)}) + Y(t), \sigma_X(t) dW_X(t)\rangle + \frac{1}{2}\trace(\sigma_X(t)\sigma_X(t)^*)dt
    \\&=\Bigg(\frac{\dot{\varepsilon}(t)}{2}\|X(t)\|^2 - \frac{\dot{\varepsilon}(t)}{2}\|x_{\varepsilon(t)}\|^2
    \\&\qquad+ \left\langle \lambda\sqrt{\varepsilon(t)}(X(t)-x_{\varepsilon(t)}) + Y(t), \lambda\frac{\dot{\varepsilon}(t)}{2\sqrt{\varepsilon(t)}}(X(t)-x_{\varepsilon(t)}) - \lambda\sqrt{\varepsilon(t)}\frac{d}{dt}x_{\varepsilon(t)}\right\rangle
    \\&\qquad + \langle \nabla f(X(t)) + \varepsilon(t)X(t) + \lambda^2\varepsilon(t)(X(t)-x_{\varepsilon(t)}) + \lambda\sqrt{\varepsilon(t)}Y(t), Y(t)\rangle
    \\&\qquad+ \langle \lambda\sqrt{\varepsilon(t)}(X(t)-x_{\varepsilon(t)}) + Y(t), -\delta \sqrt{\varepsilon(t)}Y(t) - \nabla f(X(t)) - \varepsilon(t)X(t)\rangle \Bigg) dt
    \\&\quad + \langle \lambda\sqrt{\varepsilon(t)}(X(t)-x_{\varepsilon(t)}) + Y(t), \sigma_X(t) dW_X(t)\rangle + \frac{1}{2}\trace(\sigma_X(t)\sigma_X(t)^*)dt
    \\&=\Bigg(\frac{\dot{\varepsilon}(t)}{2}\|X(t)\|^2 - \frac{\dot{\varepsilon}(t)}{2}\|x_{\varepsilon(t)}\|^2 + \frac{\lambda^2}{2}\dot{\varepsilon}(t)\|X(t)-x_{\varepsilon(t)}\|^2
    \\&\qquad+ \lambda\left(\frac{\dot{\varepsilon}(t)}{2\sqrt{\varepsilon(t)}} + (\lambda-\delta)\varepsilon(t)\right)\langle X(t)-x_{\varepsilon(t)}, Y(t)\rangle + (\lambda-\delta)\sqrt{\varepsilon(t)}\|Y(t)\|^2
    \\&\qquad+ \lambda\underbrace{\left(- \sqrt{\varepsilon(t)}\right)\langle \nabla f(X(t)) + \varepsilon(t)X(t), X(t)-x_{\varepsilon(t)}\rangle}_{=:D_0(t)}
    \\&\qquad- \lambda^2\varepsilon(t)\left\langle \frac{d}{dt}x_{\varepsilon(t)}, X(t)-x_{\varepsilon(t)}\right\rangle - \lambda\sqrt{\varepsilon(t)}\left\langle \frac{d}{dt}x_{\varepsilon(t)}, Y(t)\right\rangle\Bigg) dt
    \\&\quad + \langle \lambda\sqrt{\varepsilon(t)}(X(t)-x_{\varepsilon(t)}) + Y(t), \sigma_X(t) dW_X(t)\rangle + \frac{1}{2}\trace(\sigma_X(t)\sigma_X(t)^*)dt.
\end{align*}
Using that $\varphi_t$ is $\varepsilon(t)$-strongly convex, we have
\begin{align*}
    \varphi_t(x_{\varepsilon(t)}) - \varphi_t(X(t)) \geq \langle \nabla \varphi_t(X(t)), x_{\varepsilon(t)} - X(t)\rangle + \frac{\varepsilon(t)}{2}\|X(t) - x_{\varepsilon(t)}\|^2.
\end{align*}
Since $\varepsilon$ is nonincreasing, so $\dot{\varepsilon}(t)\leq 0$, we can incorporate the above estimation into $D_0(t)$ to obtain
\begin{align}\label{eq:estimate-D0}
    D_0(t)\leq \left( - \sqrt{\varepsilon(t)}\right)(\varphi_t(X(t)) - \varphi_t(x_{\varepsilon(t)})) + \frac{1}{2}\left(- \sqrt{\varepsilon(t)}\right)\varepsilon(t)\|X(t)-x_{\varepsilon(t)}\|^2.
\end{align}
Next, introduce $a>1$ and $b:=\frac{c}{2}\lambda >0$, where $c>2$ (they will be adjusted to suitable values later) and observe
\begin{align}
    &-\lambda\sqrt{\varepsilon(t)}\left\langle \frac{d}{dt}x_{\varepsilon(t)}, Y(t)\right\rangle \leq \frac{\lambda\sqrt{\varepsilon(t)}}{2a}\|Y(t)\|^2 + \frac{a\lambda\sqrt{\varepsilon(t)}}{2}\left\|\frac{d}{dt}x_{\varepsilon(t)}\right\|^2 \ \mbox{and} \label{eq:introduce-a}
    \\&-\lambda^2\varepsilon(t)\left\langle X(t)-x_{\varepsilon(t)}, \frac{d}{dt}x_{\varepsilon(t)}\right\rangle \leq \frac{c\lambda\sqrt{\varepsilon(t)}}{2}\left\|\frac{d}{dt}x_{\varepsilon(t)}\right\|^2 + \frac{\lambda^3\varepsilon(t)^\frac{3}{2}}{2c}\|X(t)-x_{\varepsilon(t)}\|^2.\label{eq:introduce-c}
\end{align}
Next, combine~\eqref{eq:introduce-a}-\eqref{eq:introduce-c} with the above expression for $d\mathcal{E}(t,X(t),Y(t))$. Then,
\begin{align}\label{eq:estimate-d-energy}
    \begin{split}
    d\mathcal{E}(&t,X(t),Y(t)) \leq
    \\&\Bigg(\frac{\dot{\varepsilon}(t)}{2}\|X(t)\|^2 - \frac{\dot{\varepsilon}(t)}{2}\|x_{\varepsilon(t)}\|^2 + \frac{\lambda^2}{2}\dot{\varepsilon}(t)\|X(t)-x_{\varepsilon(t)}\|^2
    \\&\qquad+ \lambda\left(\frac{\dot{\varepsilon}(t)}{2\sqrt{\varepsilon(t)}} + (\lambda-\delta)\varepsilon(t)\right)\langle X(t)-x_{\varepsilon(t)}, Y(t)\rangle + (\lambda-\delta)\sqrt{\varepsilon(t)}\|Y(t)\|^2
    \\&\qquad+ \lambda\left( - \sqrt{\varepsilon(t)}\right)(\varphi_t(X(t)) - \varphi_t(x_{\varepsilon(t)})) + \frac{\lambda}{2}\left(-\sqrt{\varepsilon(t)}\right)\varepsilon(t)\|X(t)-x_{\varepsilon(t)}\|^2
    \\&\qquad + \frac{a\lambda\sqrt{\varepsilon(t)}}{2}\left\|\frac{d}{dt}x_{\varepsilon(t)}\right\|^2 +\frac{c\lambda\sqrt{\varepsilon(t)}}{2}\left\|\frac{d}{dt}x_{\varepsilon(t)}\right\|^2 + \frac{\lambda^3\varepsilon(t)^\frac{3}{2}}{2c}\|X(t)-x_{\varepsilon(t)}\|^2 + \frac{\lambda\sqrt{\varepsilon(t)}}{2a}\|Y(t)\|^2 \Bigg) dt
    \\&\quad + \langle \lambda\sqrt{\varepsilon(t)}(X(t)-x_{\varepsilon(t)}) + Y(t) , \sigma_X(t) dW_X(t)\rangle + \frac{1}{2}\trace(\sigma_X(t)\sigma_X(t)^*)dt
    \\&\leq \Bigg(\frac{1}{2}\dot{\varepsilon}(t)\|X(t)\|^2 -\frac{1}{2}\dot{\varepsilon}(t)\|x_{\varepsilon(t)}\|^2 + \lambda\left(\frac{\dot{\varepsilon}(t)}{2\sqrt{\varepsilon(t)}} + (\lambda-\delta)\sqrt{\varepsilon(t)}\right)\langle X(t)-x_{\varepsilon(t)}, Y(t)\rangle
    \\&\qquad- \lambda\sqrt{\varepsilon(t)}(\varphi_t(X(t)) - \varphi_t(x_{\varepsilon(t)})) + \left(\frac{\lambda^2}{2}\dot{\varepsilon}(t) + \frac{\lambda^3\varepsilon(t)^\frac{3}{2}}{2c} - \frac{\lambda\varepsilon(t)^{\frac{3}{2}}}{2}\right)\|X(t)-x_{\varepsilon(t)}\|^2
    \\&\qquad+ \frac{1}{2}\left(\left(2+\frac{1}{a}\right)\lambda-2\delta\right)\sqrt{\varepsilon(t)}\|Y(t)\|^2 + \frac{1}{2}(a+c)\lambda\sqrt{\varepsilon(t)}\left\|\frac{d}{dt}x_{\varepsilon(t)}\right\|^2\Bigg)dt
    \\&\quad + \langle \lambda\sqrt{\varepsilon(t)}(X(t)-x_{\varepsilon(t)}) + Y(t), \sigma_X(t) dW_X(t)\rangle + \frac{1}{2}\trace(\sigma_X(t)\sigma_X(t)^*)dt
    \end{split}
\end{align}
In the following, we will be employing $d\mathcal{E}(t,X(t),Y(t))$ as a shorthand for the RHS of~\eqref{eq:estimate-d-energy}. Now, we will turn to majorizing the expression $d\mathcal{E}(t,X(t),Y(t)) + \mu(t)\mathcal{E}(t,X(t),Y(t))dt$ with the aim of using a Gronwall-type result to arrive at our original claim. Begin with
\begin{align*}
    d\mathcal{E}(&t,X(t),Y(t)) + \mu(t)\mathcal{E}(t,X(t),Y(t))dt
    \\&\leq \Bigg(\sqrt{\varepsilon(t)}\left(-\frac{\dot{\varepsilon}(t)}{2\varepsilon(t)^\frac{3}{2}}+\delta-2\lambda\right)(\varphi_t(X(t))-\varphi_t(x_{\varepsilon(t)})) + \frac{1}{2}\dot{\varepsilon}(t)\|X(t)\|^2 - \frac{1}{2}\dot{\varepsilon}(t)\|x_{\varepsilon(t)}\|^2
    \\&\qquad + \left(\underbrace{\frac{\lambda^2}{4}\dot{\varepsilon}(t) + \frac{\lambda^3\varepsilon(t)^\frac{3}{2}}{2c} - \frac{\lambda\varepsilon(t)^\frac{3}{2}}{2} + \frac{(\delta-\lambda)\lambda^2\varepsilon(t)^\frac{3}{2}}{2}}_{\leq \frac{\lambda^3\varepsilon(t)^\frac{3}{2}}{2c} - \frac{\lambda\varepsilon(t)^\frac{3}{2}}{2} + \frac{(\delta-\lambda)\lambda^2\varepsilon(t)^\frac{3}{2}}{2}}\right)\|X(t)-x_{\varepsilon(t)}\|^2
    \\&\qquad+ \frac{1}{2}\left(\left(1+\frac{1}{a}\right)\lambda-\delta\right)\sqrt{\varepsilon(t)}\|Y(t)\|^2 - \frac{\dot{\varepsilon}(t)}{4\varepsilon(t)}\|Y(t)\|^2 + \frac{1}{2}\left(a+c\right)\lambda\sqrt{\varepsilon(t)}\left\|\frac{d}{dt}x_{\varepsilon(t)}\right\|^2\Bigg)dt
    \\&\quad + \langle \lambda\sqrt{\varepsilon(t)}(X(t)-x_{\varepsilon(t)}) + Y(t), \sigma_X(t) dW_X(t)\rangle + \frac{1}{2}\trace(\sigma_X(t)\sigma_X(t)^*)dt,
\end{align*}
which again uses that fact that $\dot{\varepsilon}(t)\leq 0$ for all $t\geq t_0$. This yields
\begin{align*}
    d\mathcal{E}(&t,X(t),Y(t)) + \mu(t)\mathcal{E}(t,X(t),Y(t))dt
    \\&\leq \Bigg(\sqrt{\varepsilon(t)}\left(-\frac{\dot{\varepsilon}(t)}{2\varepsilon(t)^\frac{3}{2}} + \delta-2\lambda \right)(\varphi_t(X(t))-\varphi(x_{\varepsilon(t)})) + \frac{1}{2}\dot{\varepsilon}(t)\|X(t)\|^2 - \frac{1}{2}\dot{\varepsilon}(t))\|x_{\varepsilon(t)}\|^2
    \\&\qquad+ \frac{\lambda \varepsilon(t)^\frac{3}{2}}{2}\left(\delta\lambda-\lambda^2-1 + \frac{\lambda}{c}\right)\|X(t)-x_{\varepsilon(t)}\|^2
    \\&\qquad+ \left(\frac{1}{2}\left(\left(1+\frac{1}{a}\right)\lambda-\delta\right)\sqrt{\varepsilon(t)}-\frac{\dot{\varepsilon}(t)}{4\varepsilon(t)}\right)\|Y(t)\|^2 + \frac{1}{2}\left(a+c\right)\lambda\sqrt{\varepsilon(t)}\left\|\frac{d}{dt}x_{\varepsilon(t)}\right\|^2\Bigg)dt
    \\&\quad + \langle \lambda\sqrt{\varepsilon(t)}(X(t)-x_{\varepsilon(t)}) + Y(t), \sigma_X(t) dW_X(t)\rangle + \frac{1}{2}\trace(\sigma_X(t)\sigma_X(t)^*)dt.
\end{align*}
By Lemma~\ref{lem:estimate-x-epsilon},
\begin{align*}
    \left\|\frac{d}{dt}x_{\varepsilon(t)}\right\|^2\leq \frac{\dot{\varepsilon}(t)^2}{\varepsilon(t)^2}\|x_{\varepsilon(t)}\|^2\leq \frac{\dot{\varepsilon}(t)^2}{\varepsilon(t)^2}\|x^*\|^2,
\end{align*}
and thus
\begin{align}\label{eq:estimate-what-will-be-in-gronwall}
\begin{split}
    d\mathcal{E}(&t,X(t),Y(t))+\mu(t)\mathcal{E}(t,X(t),Y(t))dt
    \\&\leq \Bigg(\sqrt{\varepsilon(t)}\left(-\frac{\dot{\varepsilon}(t)}{2\varepsilon(t)^\frac{3}{2}} + \delta-2\lambda \right)(\varphi_t(X(t))-\varphi_t(x_{\varepsilon(t)}))
    \\&\qquad+ \frac{1}{2}\dot{\varepsilon}(t)\|X(t)\|^2 +\frac{\lambda \varepsilon(t)^\frac{3}{2}}{2}\left(\delta\lambda-\lambda^2-1 + \frac{\lambda}{c} \right)\|X(t)-x_{\varepsilon(t)}\|^2
    \\&\qquad+\left(\frac{1}{2}\left(\left(1+\frac{1}{a}\right)\lambda-\delta\right)\sqrt{\varepsilon(t)}-\frac{\dot{\varepsilon}(t)}{2\varepsilon(t)}\right)\|Y(t)\|^2
    \\&\qquad+ \frac{1}{2}\left(\left(a+c\right)\lambda\frac{\dot{\varepsilon}(t)^2}{\varepsilon(t)^\frac{3}{2}}-\dot{\varepsilon}(t)\right)\|x_{\varepsilon(t)}\|^2\Bigg)dt
    \\&\quad + \langle \lambda\sqrt{\varepsilon(t)}(X(t)-x_{\varepsilon(t)}) + Y(t), \sigma_X(t) dW_X(t)\rangle + \frac{1}{2}\trace(\sigma_X(t)\sigma_X(t)^*) dt.
\end{split}
\end{align}
In the next step, examine the signs of the coefficients in~\eqref{eq:estimate-what-will-be-in-gronwall}.
\begin{align}\label{eq:estimate-what-will-be-in-gronwall-with-coeff-names}
\begin{split}
    d\mathcal{E}(&t,X(t),Y(t)) + \mu(t)\mathcal{E}(t,X(t),Y(t)) dt
    \\&\leq \Bigg(\sqrt{\varepsilon(t)}\left(\underbrace{-\frac{\dot{\varepsilon}(t)}{2\varepsilon(t)^\frac{3}{2}} + \delta-2\lambda}_{=:A} \right)(\varphi_t(X(t)) - \varphi_t(x_{\varepsilon(t)}))
    \\&\qquad+ \frac{1}{2}\left(\underbrace{\dot{\varepsilon}(t)}_{\leq 0}\right)\|X(t)\|^2 + \frac{\lambda}{2}\left(\underbrace{\left(\delta+\frac{1}{c}\right)\lambda - \lambda^2-1}_{=:C}\right)\varepsilon(t)^\frac{3}{2} \|X(t)-x_{\varepsilon(t)}\|^2
    \\&\qquad+\left(\underbrace{\frac{1}{2}\left(\left(1+\frac{1}{a}\right)\lambda-\delta\right)\sqrt{\varepsilon(t)}-\frac{\dot{\varepsilon}(t)}{2\varepsilon(t)}}_{=:D(t)}\right)\|Y(t)\|^2 + \frac{1}{2}\left((a+c)\lambda\frac{\dot{\varepsilon}(t)^2}{\varepsilon(t)^\frac{3}{2}} - \dot{\varepsilon}(t) \right)\|x_{\varepsilon(t)}\|^2\Bigg)dt
    \\&\quad + \langle \lambda\sqrt{\varepsilon(t)}(X(t)-x_{\varepsilon(t)}) + Y(t), \sigma_X(t) dW_X(t)\rangle + \frac{1}{2}\trace(\sigma_X(t)\sigma_X(t)^*)dt.
\end{split}
\end{align}
Now, by~\eqref{eq:condition-tikhonov-parameter},
\begin{align*}
    \frac{d}{dt}\left(\frac{1}{\sqrt{\varepsilon(t)}}\right)\leq 2\lambda - \delta \ \mbox{and} \ \frac{d}{dt}\left(\frac{1}{\sqrt{\varepsilon(t)}}\right)\leq \frac{1}{2}\left(\delta-\left(1+\frac{1}{a}\right)\lambda\right).
\end{align*}
Thus,
\begin{align*}
    A = -\frac{\dot{\varepsilon}(t)}{2\varepsilon(t)^\frac{3}{2}} + \delta-2\lambda = \frac{d}{dt}\left(\frac{1}{\sqrt{\varepsilon(t)}}\right) + \delta-2\lambda \leq 0.
\end{align*}
In the case where $\delta \leq \sqrt{2}-\frac{1}{c}$, we have
\begin{align*}
    \left(\delta+\frac{1}{c}\right)\lambda - \lambda^2 - 1 \leq \sqrt{2}\lambda - \lambda^2 - 1 = -\left(\frac{1}{\sqrt{2}}\lambda-1\right)^2 - \frac{1}{2}\lambda^2 \leq 0,
\end{align*}
and if $\delta > \sqrt{2} - \frac{1}{c}$, we have
\begin{align*}
    \left(\delta+\frac{1}{c}\right)\lambda - \lambda^2 - \frac{1}{2} \leq 0, \ \mbox{because} \ \lambda \geq \frac{1}{2}\left(\delta+\frac{1}{c} + \sqrt{\left(\delta+\frac{1}{c}\right)^2-2}\right),
\end{align*}
from which we obtain that
\begin{align*}
    C = \left(\delta+\frac{1}{c}\right)\lambda - \lambda^2 - 1\leq 0.
\end{align*}
Again applying condition~\eqref{eq:condition-tikhonov-parameter}, it follows
\begin{align*}
    D(t) = \frac{1}{2}\left(\left(1+\frac{1}{a}\right)\lambda-\delta\right) - \frac{1}{2}\frac{\dot{\varepsilon}(t)}{\varepsilon(t)^\frac{3}{2}} = \frac{d}{dt}\left(\frac{1}{\sqrt{\varepsilon(t)}}\right) + \frac{1}{2}\left(\left(1+\frac{1}{a}\right)\lambda-\delta\right)\leq 0.
\end{align*}
Thus,~\eqref{eq:estimate-what-will-be-in-gronwall-with-coeff-names} implies
\begin{align*}
    d\mathcal{E}(&t,X(t),Y(t)) + \mu(t)\mathcal{E}(t,X(t),Y(t))dt 
    \\&\leq \left(\frac{1}{2}\underbrace{\left((a+c)\lambda\frac{\dot{\varepsilon}(t)^2}{\varepsilon(t)^\frac{3}{2}} - \dot{\varepsilon}(t) \right)}_{=:G(t)}\|x_{\varepsilon(t)}\|^2\right)dt
    \\&\quad + \langle \lambda\sqrt{\varepsilon(t)}(X(t)-x_{\varepsilon(t)}) + Y(t), \sigma_X(t) dW_X(t)\rangle + \frac{1}{2}\trace(\sigma_X(t)\sigma_X(t)^*) dt.
\end{align*}
Since now $\|x_{\varepsilon(t)}\|\leq\|x^*\|$, it follows
\begin{align}\label{eq:estimate-what-will-be-in-gronwall-condensed}
\begin{split}
    d\mathcal{E}(&t,X(t),Y(t)) + \mu(t)\mathcal{E}(t,X(t),Y(t))dt
    \\&\leq \left(\frac{\|x^*\|^2}{2}G(t) \right)dt + \langle \lambda\sqrt{\varepsilon(t)}(X(t)-x_{\varepsilon(t)}) + Y(t), \sigma_X(t) dW_X(t)\rangle + \frac{1}{2}\trace(\sigma_X(t)\sigma_X(t)^*)dt.
\end{split}
\end{align}
We will use a stochastic version of the Gronwall approach, so define
\begin{align*}
    R(t,x,y) = \exp\left(\int_{t_1}^t \mu(s) ds\right)\mathcal{E}(t,x,y).
\end{align*}
Then, calling $\gamma(t):=\exp\left(\int_{t_1}^t \mu(s)ds\right)$,
\begin{align*}
    &\frac{d}{dt} R(t,x,y) = \mu(t)\gamma(t)\mathcal{E}(t,x,y) + \gamma(t)\frac{d}{dt}\mathcal{E}(t,x,y)
    \\& \frac{d}{dx}R(t,x,y) = \gamma(t)\frac{d}{dx}\mathcal{E}(t,x,y)
    \\&\frac{d}{dy}R(t,x,y) = \gamma(t)\frac{d}{dy}\mathcal{E}(t,x,y)
\end{align*}
Hence, we have, using $\frac{\lambda}{\delta}-1 \leq 0$.
\begin{align}\label{eq:gronwall-derivative-of-exponential}
\begin{split}
    dR(t,X(t),Y(t)) &= \mu(t)\gamma(t)\mathcal{E}(t,X(t),Y(t)) dt + \gamma(t) d\mathcal{E}(t,X(t),Y(t)).
    \\&\leq \frac{\|x^*\|^2}{2}G(t)\gamma(t) dt 
    \\&\quad + \gamma(t)\langle \lambda\sqrt{\varepsilon(t)}(X(t)-x_{\varepsilon(t)}) + Y(t), \sigma_X(t) dW_X(t)\rangle + \gamma(t)\frac{1}{2}\trace(\sigma_X(t)\sigma_X(t)^*) dt,
\end{split}
\end{align}
so  by taking the difference of the left and right hand sides of the inequality in~\eqref{eq:gronwall-derivative-of-exponential}, we obtain $\Delta(t)\geq 0$ such that
\begin{align}\label{eq:energy-with-exponential-equality}
\begin{split}
    R(&t,X(t),Y(t)) - R(t_1,X(t_1),Y(t_1))
    \\&= \int_{t_1}^{t}\left(\frac{\|x^*\|^2}{2}G(s)\gamma(s) - \Delta(s)\right)ds
    \\&\quad + \int_{t_1}^{t}\gamma(s)\langle \lambda\sqrt{\varepsilon(s)}(X(s)-x_{\varepsilon(s)}) + Y(s) , \sigma_X(s) dW_X(s)\rangle + \int_{t_1}^{t}\gamma(s)\frac{1}{2}\trace(\sigma_X(s)\sigma_X(s)^*) ds.
\end{split}
\end{align}
Since $R(t,X(t),Y(t)) = \gamma(t)\mathcal{E}(t,X(t),Y(t))$, this means
\begin{align*}
    \mathcal{E}(&t,X(t),Y(t))
    \\&\leq \frac{\gamma(t_1)}{\gamma(t)}\mathcal{E}(t_1,X(t_1),Y(t_1)) + \frac{1}{\gamma(t)}\int_{t_1}^t \frac{\|x^*\|^2}{2}G(s)\gamma(s) ds
    \\&\quad + \frac{1}{\gamma(t)}\int_{t_1}^{t}\gamma(s)\langle \lambda\sqrt{\varepsilon(s)}(X(s)-x_{\varepsilon(s)}) + Y(s) , \sigma_X(s) dW_X(s)\rangle + \frac{1}{\gamma(t)}\int_{t_1}^{t}\gamma(s)\frac{1}{2}\trace(\sigma_X(s)\sigma_X(s)^*) ds.
\end{align*}
Consider
\begin{align*}
    \mathbb{E}&\left(\int_{t_1}^T \|\gamma(s)\sigma(s)^*(\lambda\sqrt{\varepsilon(s)}(X(s)-x_{\varepsilon(s)}) + Y(s))\|^2ds\right)
    \\&\leq 2\gamma(T)\int_{t_1}^T\mathbb{E}\left((\lambda\sqrt{\varepsilon(s)})^2\|X(s)-x_{\varepsilon(s)}\|^2 + \|Y(s)\|^2\right) \|\sigma(s)\|^2ds
    \\&\leq 2\gamma(T)\int_{t_1}^T\mathbb{E}\left( \left(2\left(\lambda\sqrt{\varepsilon(s)}\right)^2 + L^2\right)\|X(s)\|^2 + 2(\lambda\sqrt{\varepsilon(s)})^2\|x_{\varepsilon(s)}\|^2+\|Y(s)\|^2\right) \|\sigma(s)\|^2ds
    \\&\leq 2\gamma(T)\max\Bigg(\sup_{t\in[t_1,T]}\left(2(\lambda^2\varepsilon(t) + L^2\right)\mathbb{E}\left(\sup_{t\in[t_1,T]}\|X(t)\|^2\right), 2\sup_{t\in[t_1,T]}\left((\lambda\sqrt{\varepsilon(t)})^2\|x_{\varepsilon(t)}\|^2\right),
    \\&\qquad\qquad\qquad \mathbb{E}\left(\sup_{t\in[t_1,T]}\|Y(t)\|^2\right)\Bigg) \int_{t_1}^{T} \|\sigma(s)\|^2ds < +\infty.
\end{align*}
Therefore, 
\begin{align*}
    \int_{t_1}^{t} \gamma(s) \langle \lambda\sqrt{\varepsilon(s)}(X(s)-x_{\varepsilon(s)}) + Y(s), \sigma_X(s) dW_X(s)\rangle
\end{align*}
is a square-integrable martingale with expected value 0. Using this as well as the square-integrability of $\sigma_X$, we obtain
\begin{align*}
\begin{split}
    \mathbb{E}(\mathcal{E}(&t,X(t),Y(t)))
    \\&\leq \frac{\gamma(t_1)}{\gamma(t)}\mathbb{E}(\mathcal{E}(t_1,X(t_1),Y(t_1))) + \frac{1}{\gamma(t)}\int_{t_1}^t\frac{\|x^*\|^2}{2}G(s)\gamma(s) ds + \int_{t_1}^{+\infty}\|\sigma_X(s)\|^2ds,
\end{split}
\end{align*}
or in the case where $\int_{t_1}^{+\infty} \gamma(s)\|\sigma_X(s)\|^2 < +\infty$,
\begin{align*}
    \mathbb{E}(\mathcal{E}(&t,X(t),Y(t)))
    \\&\leq \frac{\gamma(t_1)}{\gamma(t)}\mathbb{E}(\mathcal{E}(t_1,X(t_1),Y(t_1))) + \frac{1}{\gamma(t)}\int_{t_1}^t\frac{\|x^*\|^2}{2}G(s)\gamma(s) ds + \frac{1}{\gamma(t)}\int_{t_1}^{+\infty}\gamma(s)\|\sigma_X(s)\|^2ds,
\end{align*}
as desired.
\end{proof}
\begin{rmk}
    Note that in this proof, we used It\^{o}'s formula while not checking the requirements for all of $\nabla_{(x,y)} \mathcal{E}(t,(x,y))$ and $\nabla_{(x,y)}^2\mathcal{E}(t,(x,y))$ but only for those components that are not multiplied by $0$ in the integral representation of $\mathcal{E}(t,X(t),Y(t))$ -- in the proof of It\^{o}'s formula, the assumptions are only used for precisely these components, so it is sufficient to verify the requirements only on them. Incidentally, $f\in\mathcal{C}^1$ would therefore be enough for these necessary conditions of It\^{o}'s formula, however, if the system did not have the property that $dX(t)$ has diffusion term $0$, we would need $f\in\mathcal{C}^2$ here as well, since $\frac{d^2}{dx^2}\mathcal{E}(t,x,y)$ contains $\nabla^2 f(x)$ which would then need to be continuous and bounded on bounded sets. In the current setup, the need for $f\in\mathcal{C}^2$ comes from the fact that $\frac{d}{dt}x_{\varepsilon(t)}$ has to be continuous in order for us to apply It\^{o}'s formula.
\end{rmk}
Before we move to the next section, consider the adaptation of Lemma 1 from~\cite{attouch-balhag-chbani-riahi} to the stochastic world.
\begin{lem}\label{lem:for-the-rates}
    Let $X:[t_0,+\infty)\rightarrow \mathcal{H}$ be a solution trajectory of~\eqref{eq:s-TRIGS}, $x^*$ be the minimum norm minimizer of $f$ and $\mathcal{E}$ be the energy function defined in~\eqref{eq:def-energy-function}. Then, the following estimates are satisfied:
    \begin{align*}
        &f(X(t))-\inf_\mathcal{H} f \leq \mathcal{E}(t,X(t),Y(t)) + \frac{\varepsilon(t)}{2}\|x^*\|^2,
        \\& \|X(t)-x_{\varepsilon(t)}\|^2\leq \frac{2\mathcal{E}(t,X(t),Y(t))}{\varepsilon(t)}.
    \end{align*}
    Therefore, $X(t)$ converges strongly to $x^*$ almost surely if $\lim_{t\rightarrow+\infty}\frac{\mathcal{E}(t,X(t),Y(t))}{\varepsilon(t)}=0$.
\end{lem}
\begin{proof}
    By the definition of $\varphi_t$, we have
    \begin{align*}
        f(X(t))-\inf_\mathcal{H} f &= \varphi_t(X(t))-\varphi_t(x^*) + \frac{\varepsilon(t)}{2}(\|x^*\|^2-\|X(t)\|^2)
        \\&= (\varphi_t(X(t)) - \varphi_t(x_{\varepsilon(t)})) + \underbrace{(\varphi_t(x_{\varepsilon(t)}) - \varphi_t(x^*))}_{\leq 0} + \frac{\varepsilon(t)}{2}(\|x^*\|^2-\|X(t)\|^2)
        \\&\leq \varphi_t(X(t)) - \varphi_t(x_{\varepsilon(t)}) + \frac{\varepsilon(t)}{2}\|x^*\|^2.
    \end{align*}
    By definition of $\mathcal{E}(t,X(t),Y(t))$ we have
    \begin{align}\label{eq:helper-lemma}
        \varphi_t(X(t))-\varphi_t(x_{\varepsilon(t)}) \leq \mathcal{E}(t,X(t),Y(t)),
    \end{align}
    which, combined with the above inequality, yields the first claim.

    By the strong convexity of $\varphi_t$ and by $x_{\varepsilon(t)} := \argmin_\mathcal{H} \varphi_t$, we have
    \begin{align*}
        \varphi_t(X(t))-\varphi_t(x_{\varepsilon(t)})\geq \frac{\varepsilon(t)}{2}\|X(t)-x_{\varepsilon(t)}\|^2.
    \end{align*}
    Combining this inequality with~\eqref{eq:helper-lemma}, we obtain
    \begin{align*}
        \mathcal{E}(t,X(t),Y(t)) \geq \frac{\varepsilon(t)}{2}\|X(t)-x_{\varepsilon(t)}\|^2,
    \end{align*}
    giving the second claim.
\end{proof}
Note that this means that the convergence results that we will prove for the particular case $\varepsilon(t) = \frac{1}{t^r}$ in the following section actually hold for any $\varepsilon$ that satisfies $\lim_{t\rightarrow+\infty}\frac{\mathcal{E}(t,X(t),Y(t))}{\varepsilon(t)}=0$.

\section{A standard choice for the regularization parameter}\label{sec:particular-case-epsilon-1/t^r}
In this section, we consider for $0<r<2$ and $t_0>0$ the parameter function $\varepsilon(t) := \frac{1}{t^r}$. Therefore, the system reads
\begin{align*}
\begin{cases}
    dX(t) = Y(t)dt
    \\ dY(t) = (-\delta \frac{1}{t^\frac{r}{2}}Y(t) - \nabla f(X(t)) - \frac{1}{t^r} X(t))dt + \sigma_X(t) dW_X(t).
\end{cases}
\end{align*}
In this setting, we prove the following convergence rates:
\begin{thm}\label{thm:strong-rates}
    If $\int_{t_1}^{+\infty} s^r\exp\left(\frac{4(\delta-\lambda)}{2-r}s^{1-\frac{r}{2}}\right)\|\sigma_X(s)\|_{HS}^2 ds <+\infty$,
    \begin{align*}
        f(X(t)) - \inf_{\mathcal{H}} f = \mathcal{O}\left(\frac{1}{t^r}\right) \ \mbox{a.s.}, \ \|X(t)-x_{\varepsilon(t)}\|^2 = \mathcal{O}\left(\frac{1}{t^\frac{2-r}{2}}\right) \ \mbox{a.s. as} \ t\rightarrow +\infty,
    \end{align*}
    and
    \begin{align*}
        \|Y(t)\| = \begin{cases}
            \mathcal{O}\left(\frac{1}{t^\frac{r+2}{4}}\right) \ \mbox{if} \ r\in \left[\frac{2}{3}, 2\right)
            \\ \mathcal{O}\left(\frac{1}{t^r}\right) \ \mbox{if} \ r\in \left(0,\frac{2}{3}\right)
        \end{cases} \ \mbox{a.s. as} \ t\rightarrow +\infty,
    \end{align*}
    as well as the same rates in expectation:
    \begin{align*}
        \mathbb{E}(f(X(t)) - \inf_{\mathcal{H}} f) = \mathcal{O}\left(\frac{1}{t^r}\right), \ \mathbb{E}\left(\|X(t)-x_{\varepsilon(t)}\|^2\right) = \mathcal{O}\left(\frac{1}{t^\frac{2-r}{2}}\right) \ \mbox{as} \ t\rightarrow +\infty,
    \end{align*}
    and
    \begin{align*}
        \mathbb{E}(\|Y(t)\|) = \begin{cases}
            \mathcal{O}\left(\frac{1}{t^\frac{r+2}{4}}\right) \ \mbox{if} \ r\in \left[\frac{2}{3}, 2\right)
            \\ \mathcal{O}\left(\frac{1}{t^r}\right) \ \mbox{if} \ r\in \left(0,\frac{2}{3}\right)
        \end{cases} \ \mbox{as} \ t\rightarrow +\infty,
    \end{align*}
    Furthermore, $X(t)$ converges strongly to $x^*$ almost surely, where $x^*$ denotes the minimum norm minimizer of $f$.
\end{thm}
\begin{proof}
    First, observe that for the specific choice of $\varepsilon$, we have $\gamma(t) = \left(\frac{t}{t_1}\right)^\frac{r}{2}\exp\left(\frac{2(\delta-\lambda)}{2-r}\left(t^{1-\frac{r}{2}} - t_1^{1-\frac{r}{2}}\right)\right)$ and
    \begin{align*}
        \lim_{t\rightarrow +\infty} \frac{1}{\gamma(t)}\int_{t_1}^{t} G(s)\gamma(s)ds = 0,
    \end{align*}
    since
    \begin{align*}
    \begin{split}
        \frac{1}{\gamma(t)}&\int_{t_1}^{t} G(s)\gamma(s)ds
        \\&=\left(\frac{t_1}{t}\right)^{\frac{r}{2}}\exp\left(\underbrace{\frac{2}{r-2}(\delta-\lambda)\left(\frac{1}{t^\frac{r-2}{2}} - \frac{1}{t_1^\frac{r-2}{2}}\right)}_{=:g(t)}\right)
        \\&\qquad\qquad\int_{t_1}^{t}\left((a+c)\lambda r^2\frac{1}{s^{\frac{r}{2}+2}} + r\frac{1}{s^{r+1}}\right) \left(\frac{s}{t_1}\right)^\frac{r}{2} \exp\left(\frac{2}{2-r}(\delta-\lambda)\left(\frac{1}{s^\frac{r-2}{2}} - \frac{1}{t_1^\frac{r-2}{2}}\right)\right)ds
        \\&\leq \left(\frac{t_1}{t}\right)^{\frac{r}{2}}\exp(g(t))\int_{t_1}^{t}\left((a+c)\lambda r^2 s^{-\frac{r}{2}-2} + rs^{-r-1}\right)\left(\frac{s}{t_1}\right)^\frac{r}{2}\exp(-g(t))ds
        \\&= \left(\frac{1}{t}\right)^{\frac{r}{2}}\left((a+c)\lambda r^2\left(\frac{1}{t_1} - \frac{1}{t}\right) + 2\left(\frac{1}{t_1^\frac{r}{2}} - \frac{1}{t^\frac{r}{2}}\right)\right)
        \\&\leq \mathcal{O}\left(\frac{1}{t^\frac{r}{2}}\right) \rightarrow 0 \ \mbox{as} \ t\rightarrow +\infty.
    \end{split}
    \end{align*}
    Now recall that
    \begin{align*}
        \mathcal{E}(&t,X(t),Y(t))
        \\&\leq \frac{\gamma(t_1)}{\gamma(t)}\mathcal{E}(t_1,X(t_1),Y(t_1)) + \frac{1}{\gamma(t)}\int_{t_1}^t \frac{\|x^*\|^2}{2}G(s)\gamma(s) ds
        \\&\quad + \frac{1}{\gamma(t)} \int_{t_1}^{t}\gamma(s)\langle \lambda\sqrt{\varepsilon(s)}(X(s)-x_{\varepsilon(s)}) + Y(s), \sigma_X(s) dW_X(s)\rangle + \frac{1}{\gamma(t)} \int_{t_1}^{t}\frac{1}{2}\gamma(s)\trace(\sigma_X(s)\sigma_X(s)^*)ds.
    \end{align*}
    Take the parameters $a>1$, $c>2$, $\lambda>0$ such that
    \begin{align*}
        \frac{1}{2}\left(\delta+\frac{1}{c} + \sqrt{\left(\delta+\frac{1}{c}\right)^2-2}\right) < \lambda < \min\left(\frac{a}{a+1}\delta, \frac{\delta + \sqrt{\delta^2-4}}{2}\right).
    \end{align*}
    For $t\geq t_1$ large enough, we can prove that
    \begin{align*}
        \frac{d}{dt}\left(\frac{1}{\sqrt{\varepsilon(t)}}\right) = \frac{r}{2}t^\frac{r-2}{2} \leq \min\left(2\lambda-\delta, \frac{1}{2}\left(\delta-\frac{a+1}{a}\lambda\right)\right).
    \end{align*}
    Therefore, the condition~\eqref{eq:condition-tikhonov-parameter} is satisfied. In order to apply some of the estimates established in the proof of Theorem~\ref{thm:rates-in-expectation}, observe
    \begin{align*}
        \mu(t) &= -\frac{\dot{\varepsilon}(t)}{2\varepsilon(t)} + (\delta-\lambda)\sqrt{\varepsilon(t)} = \frac{r}{2t} + \frac{\delta-\lambda}{t^\frac{r}{2}}
        \\ \gamma(t) &= \exp\left(\int_{t_1}^t \mu(s)ds\right) = \left(\frac{t}{t_1}\right)^\frac{r}{2}\exp\left(\frac{2(\delta-\lambda)}{2-r}\left(t^\frac{2-r}{2} - t_1^\frac{2-r}{2}\right)\right)
        \\&= C_1t^\frac{r}{2} \exp\left(\frac{2(\delta-\lambda)}{2-r}t^\frac{2-r}{2}\right), \ \mbox{where} \ C_1 := \left(t_1^\frac{r}{2}\exp\left(\frac{2(\delta-\lambda)}{2-r}t_1^\frac{2-r}{2}\right)\right)^{-1}.
    \end{align*}
    Now, set
    \begin{align*}
        \lambda_0 := (a+c)\lambda, \ \delta_0 := \frac{2(\delta-\lambda)}{2-r}
    \end{align*}
    and combine this with the estimate for $\mathcal{E}(t,X(t),Y(t))$ to obtain
    \begin{align*}
        \mathcal{E}(t,X(t),Y(t)) &\leq \frac{r}{2t^\frac{r}{2}\exp\left(\delta_0\left(t^\frac{2-r}{2} - t_1^\frac{2-r}{2}\right)\right)}\int_{t_1}^t \left(\frac{\lambda_0r}{s^2} + \frac{1}{s^\frac{r+2}{2}}\right)\exp\left(\delta_0\left(s^\frac{2-r}{2} - t_1^\frac{2-r}{2}\right)\right) ds
        \\&\quad+ \frac{\gamma(t_1)\mathcal{E}(t_1,X(t_1),Y(t_1))}{\gamma(t)}
        \\&\quad + \int_{t_1}^{t}\langle \lambda\sqrt{\varepsilon(s)}(X(s)-x_{\varepsilon(s)}) + Y(s), \sigma_X(s) dW_X(s)\rangle + \frac{1}{\gamma(t)} \int_{t_1}^{t}\frac{1}{2}\gamma(s)\trace(\sigma_X(s)\sigma_X(s)^*)ds.
    \end{align*}
    Next, consider the integral $\int_{t_1}^t \left(\frac{\lambda_0 r}{s^2} + \frac{1}{s^\frac{r+2}{2}}\right)\exp\left(\delta_0\left(s^\frac{2-r}{2} - t_1^\frac{2-r}{2}\right)\right)ds$. For $\rho>0$,
    \begin{align*}
        \frac{d}{ds}\left(\frac{1}{\rho s}\exp\left(\delta_0\left(s^\frac{2-r}{2} - t_1^\frac{2-r}{2}\right)\right)\right) = \left(-\frac{1}{\rho s^2} + \frac{\delta_0(2-r)}{2\rho s^\frac{r+2}{2}}\right)\exp\left(\delta_0\left(s^\frac{2-r}{2} - t_1^\frac{2-r}{2}\right)\right).
    \end{align*}
    We now want to show that
    \begin{align*}
        \frac{\lambda_0 r}{s^2} + \frac{1}{s^\frac{r+2}{2}} \leq -\frac{1}{\rho s^2} + \frac{\delta_0(2-r)}{2\rho s^\frac{r+2}{2}}.
    \end{align*}
    By taking $\rho < \frac{1}{(a+1)}\delta$, we have
    \begin{align}\label{eq:helper-strong-rates}
    \begin{split}
        \frac{\lambda_0 r}{s^2} + \frac{1}{s^\frac{r+2}{2}} \leq -\frac{1}{\rho s^2} + \frac{\delta_0(2-r)}{2\rho s^\frac{r+2}{2}} &\Leftrightarrow \frac{\lambda_0 r + \frac{1}{\rho}}{s^2} \leq \left(\frac{\delta_0 (2-r)}{2\rho} - 1\right)\frac{1}{s^\frac{r+2}{2}} = \frac{\frac{\delta-\lambda}{\rho} - 1}{s^\frac{r+2}{2}}
        \\&\Leftrightarrow \frac{\lambda_0 r + \frac{1}{\rho}}{s^\frac{2-r}{2}} \leq \frac{\delta-\lambda}{\rho} - 1.
    \end{split}
    \end{align}
    Since $0<r<2$, it follows that $\lim_{s\rightarrow + \infty} \frac{1}{s^\frac{2-r}{2}} = 0$. The fact that $\lambda < \frac{a}{a+1}\delta$ in conjunction with the choice of $\rho$ implies
    \begin{align*}
        \delta-\lambda-\rho > \underbrace{\left(\frac{a}{a+1}\delta-\lambda\right)}_{>0} + \frac{1}{a+1}\delta - \rho > -\left(\rho - \frac{1}{a+1}\delta\right)>0.
    \end{align*}
    Thus, the last inequality in the equivalence~\eqref{eq:helper-strong-rates} holds for sufficiently large $s$, so for $1\leq r < 2$ and $t_1$ large enough, we have
    \begin{align*}
        \mathcal{E}(t,X(t),Y(t)) &\leq \frac{r}{2t^\frac{r}{2}\exp\left(\delta_0\left(t^\frac{2-r}{2} - t_1^\frac{2-r}{2}\right)\right)}\int_{t_1}^t \left(-\frac{1}{\rho s^2} + \frac{\delta_0(2-r)}{2\rho s^\frac{r+2}{2}}\right)\exp\left(\delta_0\left(s^\frac{2-r}{2} - t_1^\frac{2-r}{2}\right)\right) ds
        \\&\quad+ \frac{\gamma(t_1)\mathcal{E}(t_1,X(t_1),Y(t_1))}{\gamma(t)}
        \\&\quad + \frac{1}{\gamma(t)}\int_{t_1}^{t}\gamma(s)\langle \lambda\sqrt{\varepsilon(s)}(X(s)-x_{\varepsilon(s)}) + Y(s), \sigma_X(s) dW_X(s)\rangle
        \\&\quad+ \frac{1}{\gamma(t)} \int_{t_1}^{t}\frac{1}{2}\gamma(s)\trace(\sigma_X(s)\sigma_X(s)^*)ds
        \\&= \frac{r}{2t^\frac{r}{2}\exp\left(\delta_0\left(t^\frac{2-r}{2} - t_1^\frac{2-r}{2}\right)\right)}\int_{t_1}^t \frac{d}{ds}\left(\frac{1}{\rho s}\exp\left(\delta_0\left(s^\frac{2-r}{2} - t_1^\frac{2-r}{2}\right)\right)\right)ds
        \\&\quad+ \frac{\gamma(t_1)\mathcal{E}(t_1,X(t_1),Y(t_1))}{\gamma(t)}
        \\&\quad + \frac{1}{\gamma(t)}\int_{t_1}^{t}\gamma(s)\langle \lambda\sqrt{\varepsilon(s)}(X(s)-x_{\varepsilon(s)}) + Y(s), \sigma_X(s) dW_X(s)\rangle
        \\&\quad+ \frac{1}{\gamma(t)} \int_{t_1}^{t}\frac{1}{2}\gamma(s)\trace(\sigma_X(s)\sigma_X(s)^*)ds
        \\&= \frac{r}{2\rho t^\frac{r+2}{2}} - \frac{r}{t^\frac{r}{2}\exp\left(\delta_0 t^\frac{2-r}{2}\right)}\frac{1}{2\rho t_1}\exp\left(\delta_0t_1^\frac{2-r}{2}\right) + \frac{\gamma(t_1)\mathcal{E}(t_1,X(t_1),Y(t_1))}{\gamma(t)}
        \\&\quad + \frac{1}{\gamma(t)}\int_{t_1}^{t}\gamma(s)\langle \lambda\sqrt{\varepsilon(s)}(X(s)-x_{\varepsilon(s)}) + Y(s), \sigma_X(s) dW_X(s)\rangle
        \\&\quad+ \frac{1}{\gamma(t)} \int_{t_1}^{t}\frac{1}{2}\gamma(s)\trace(\sigma_X(s)\sigma_X(s)^*)ds
    \end{align*}
Therefore, we can denote the difference between the left- and right-hand sides of the above inequality by $U(t)$ and obtain the identity
\begin{align}\label{eq:estimate-power-of-t-times-energy}
\begin{split}
    t^\frac{r+2}{2}\mathcal{E}(t,&X(t),Y(t)) - \frac{t^\frac{r+2}{2}\gamma(t_1)\mathcal{E}(t_1, X(t_1),Y(t_1))}{\gamma(t)} - \frac{t^\frac{r+2}{2}}{\gamma(t)} \int_{t_1}^{t}\frac{1}{2}\gamma(s)\trace(\sigma_X(s)\sigma_X(s)^*)ds
    \\&= \frac{r}{2\rho} + \frac{t^\frac{r+2}{2}}{\gamma(t)}\int_{t_1}^t \gamma(s)\left\langle \lambda \sqrt{\varepsilon(s)}(X(s)-x_{\varepsilon(s)}) + Y(s), \sigma_X(s) dW_X(s)\right\rangle - t^\frac{r+2}{2}U(t).
\end{split}
\end{align}
Because $U(t)\geq 0$ is given by the integral of the difference of the integrands of the left- and right-hand sides and remains nonnegative for all $t$ for any choice of $t_1$, the integrand of $U(t)$ must be nonnegative, thus $U(t)$ is increasing, from which we also get that $t^\frac{r+2}{2}U(t)$ must be increasing. Since the first line of the RHS is actually a constant, we only need to consider
\begin{align*}
    \frac{t^{\frac{r+2}{2}}}{\gamma(t)}\int_{t_1}^{t} \gamma(s)\left\langle \lambda\sqrt{\varepsilon(s)}(X(s)-x_{\varepsilon(s)}) + Y(s), \sigma_X(s) dW_X(s)\right\rangle
\end{align*}
in order to be able to apply Theorem~\ref{thm:A.9-in-paper}.

Recall
\begin{align*}
    \gamma(t) = \left(\frac{1}{\sqrt{\varepsilon(t)}}\right)^\frac{r}{2}\frac{1}{t_1^\frac{r}{2}}\exp\left(\frac{2(\delta-\lambda)}{2-r}(t^{1-\frac{r}{2}}-t_1^{1-\frac{r}{2}}) \right),
\end{align*}
so
\begin{align*}
    &\sup_{t_1\leq r < +\infty} \mathbb{E}\left(\left|\int_{t_1}^{r} \gamma(s)\left\langle \lambda\sqrt{\varepsilon(s)}(X(s)-x_{\varepsilon(s)}) + Y(s), \sigma_X(s) dW_X(s)\right\rangle \right| \right)
    \\&\leq \sup_{t_1\leq r < +\infty} \mathbb{E}\left(\left(\int_{t_1}^{r} \gamma(s) \left\langle \sigma_X(s)^*\left(\lambda\sqrt{\varepsilon(s)}(X(s)-x_{\varepsilon(s)}) + Y(s) \right), dW_X(s) \right\rangle \right)^2 \right)^\frac{1}{2}
    \\&= \sup_{t_1\leq r < +\infty} \mathbb{E}\left(\int_{t_1}^{r} \gamma(s)^2\left\|\sigma_X(s) \left(\lambda\sqrt{\varepsilon(s)}(X(s)-x_{\varepsilon(s)}) + Y(s) \right) \right\|^2ds \right)^\frac{1}{2}
    \\&\leq \sup_{t_1\leq r < +\infty} \mathbb{E}\left(\int_{t_1}^{r} 2\gamma(s)^2\|\sigma_X(s)\|_{HS}^2\left(\lambda^2 \varepsilon(s)\|X(s)-x_{\varepsilon(s)}\|^2 + \|Y(s)\|^2\right) ds \right)^\frac{1}{2}
    \\&\leq \left(\int_{t_1}^{+\infty} 2\lambda^2\gamma(s)^2\|\sigma_X(s)\|_{HS}^2ds \right)^\frac{1}{2}\left(\sup_{t_1\leq s < +\infty}\mathbb{E}\left(\|X(s)-x_{\varepsilon(s)}\|^2\right) +\sup_{t_1\leq s < +\infty}\mathbb{E}\left(\|Y(s) \|^2\right)\right)^\frac{1}{2}
    \\&< +\infty,
\end{align*}
which implies using Doob's martingale convergence theorem that there exists an a.s. finite limit
\begin{align*}
    \int_{t_1}^{+\infty} \gamma(s)\left\langle \lambda\sqrt{\varepsilon(s)}(X(s)-x_{\varepsilon(s)}) + Y(s), \sigma_X(s) dW_X(s)\right\rangle < +\infty,
\end{align*}
and thus
\begin{align*}
    \frac{t^{\frac{r+2}{2}}}{\gamma(t)}\int_{t_1}^{t} \gamma(s)\left\langle \lambda\sqrt{\varepsilon(s)}(X(s)-x_{\varepsilon(s)}) + Y(s), \sigma_X(s) dW_X(s)\right\rangle \rightarrow 0 \ \mbox{a.s. as} \ t\rightarrow +\infty.
\end{align*}
Therefore, we can write~\eqref{eq:estimate-power-of-t-times-energy} as
\begin{align*}
    t^\frac{r+2}{2}\mathcal{E}(t,&X(t),Y(t)) - \frac{t^\frac{r+2}{2}\gamma(t_1)\mathcal{E}(t_1, X(t_1),Y(t_1))}{\gamma(t)} - \frac{t^\frac{r+2}{2}}{\gamma(t)} \int_{t_1}^{t}\frac{1}{2}\gamma(s)\trace(\sigma_X(s)\sigma_X(s)^*)ds
    \\&\qquad- \frac{t^\frac{r+2}{2}}{\gamma(t)}\int_{t_1}^t \gamma(s)\left\langle \lambda \sqrt{\varepsilon(s)}(X(s)-x_{\varepsilon(s)}) + Y(s), \sigma_X(s) dW_X(s)\right\rangle
    \\&= \frac{r}{2\rho} - t^\frac{r+2}{2}U(t),
\end{align*}
where the RHS now fits the framework of Theorem~\ref{thm:A.9-in-paper}. Using it yields a.s. strong convergence of $t^\frac{r+2}{2}\mathcal{E}(t,X(t),Y(t))$, i.e. that $\mathcal{E}(t,X(t),Y(t)) = \mathcal{O}\left(\frac{1}{t^\frac{r+2}{2}}\right)$ a.s., since the second through fourth summands of the LHS go to $0$ with an exponential rate. For the second summand, this is immediate, and for the fourth summand, this was discussed above. For the third summand, observe that from the definition of the Hilbert-Schmidt norm,
\begin{align*}
    \trace(\sigma_X(s)\sigma_X(s)^*) = \sum_{i} |\langle e_i, \sigma_X(s)\sigma_X(s)^*e_i\rangle | = \|\sigma_X(s)^*\|_{HS}^2 = \|\sigma_X(s)\|_{HS}^2.
\end{align*}
Therefore, we can employ the assumption that $\int_{t_1}^{+\infty} \gamma(s) \|\sigma_X(s)\|_{HS}^2 < +\infty$ to obtain the exponential convergence to $0$ of the third summand.

From $\mathcal{E}(t,X(t),Y(t)) = \mathcal{O}\left(\frac{1}{t^\frac{r+2}{2}}\right)$ a.s., it follows on the one hand by Lemma~\ref{lem:for-the-rates} that almost surely, $X(t)$ converges strongly to a solution $x^*$ of the original problem. On the other hand, this also implies the convergence rate of
\begin{align*}
    f(X(t)) - \inf_{\mathcal{H}} f = \mathcal{O}\left(\frac{1}{t^r}\right) \ \mbox{a.s.}
\end{align*}
Furthermore, the estimates from Lemma~\ref{lem:for-the-rates} imply the existence of constants $M$ and $C$ such that, a.s.,
\begin{align*}
    f(X(t)) - \inf_{\mathcal{H}} f \leq C\left(\frac{1}{t^\frac{r+2}{2}} + \frac{1}{t^r}\right)&, \ \|X(t) - x_{\varepsilon(t)}\|^2 \leq \frac{2\mathcal{E}(t,X(t),Y(t))}{\varepsilon(t)} \leq \frac{2C}{t^\frac{2-r}{2}},
    \\ \|Y(t) - \varepsilon(t)X(t)\|^2 &\leq M\mathcal{E}(t,X(t),Y(t))\leq \frac{MC}{t^\frac{r+2}{2}}.
\end{align*}
Since $0<r<2$, it follows
\begin{align*}
    f(X(t)) - \inf_{\mathcal{H}} f = \mathcal{O}\left(\frac{1}{t^r}\right) \ \mbox{a.s.}, \ \|X(t)-x_{\varepsilon(t)}\|^2 = \mathcal{O}\left(\frac{1}{t^\frac{2-r}{2}}\right) \ \mbox{a.s. as} \ t\rightarrow +\infty.
\end{align*}
For the final statement, consider
\begin{align*}
    \|Y(t)\| \leq \|Y(t) - \varepsilon(t)X(t)\| + \varepsilon(t)\|X(t)\| \leq \frac{\sqrt{MC}}{t^\frac{r+2}{4}} + \frac{1}{t^r}\|X(t)\|.
\end{align*}
Since $X$ is bounded a.s., we can conclude that a.s.,
\begin{align*}
    \|Y(t)\| = \begin{cases}
        \mathcal{O}\left(\frac{1}{t^\frac{r+2}{4}}\right) \ \mbox{if} \ r\in \left[\frac{2}{3}, 2\right)
        \\ \mathcal{O}\left(\frac{1}{t^r}\right) \ \mbox{if} \ r\in \left(0,\frac{2}{3}\right).
    \end{cases}
\end{align*}
Taking the expected value of~\eqref{eq:estimate-power-of-t-times-energy} yields boundedness of the expected value of its LHS and thus shows
\begin{align*}
    \mathbb{E}(\mathcal{E}(t,X(t),Y(t)) = \mathcal{O}\left(\frac{1}{t^{\frac{r+2}{2}}}\right),
\end{align*}
from which analogous arguments to the a.s. rates from above show the same behaviour for the expected values.
\end{proof}
\begin{rmk}
    Consider the system incorporating the more general parameter functions $\frac{\alpha}{t^q}$ and $\frac{a}{t^p}$, where $\alpha,q,a,p>0$:
\begin{align}\label{eq:S-TRIGS-GEN}\tag{S-TRIGS-GEN}
    \begin{cases}
        dX(t) = Y(t)dt
        \\ dY(t) = \left(-\frac{\alpha}{t^q}Y(t) - \nabla f(X(t)) - \frac{a}{t^p}X(t) \right)dt + \sigma_X(t) dW_X(t)
        \\ X(t_0) = X_0, \ Y(t_0) = Y_0,
    \end{cases}
\end{align}
    i.e. being no longer bound by the condition that $p = 2q$. For this system, it is possible to prove a.s. convergence rates as well as convergence rates in expectation for the function values along the trajectory to an infimal value and the time derivative of the trajectory process to zero under suitable integrability conditions on $\sigma_X$. This is a stochastic version of the system examined in~\cite{laszlo}, in which analogous results are proved for the deterministic setting.
\end{rmk}

\section*{Acknowledgements}
The author would like to thank Radu Ioan Bo\cb{t} for the fruitful discussions and valuable feedback on a previous draft of this paper.

\renewcommand\thesection{\Alph{section}}
\renewcommand\thesubsection{\thesection.\Alph{subsection}}
\setcounter{section}{0}

\section{Appendix}
Consider the general formulation of a stochastic differential equation on $[t_0, +\infty]$, where $t_0 \geq 0$,
    \begin{align*}
    \tag{SDEGEN}
    \label{SDE-GEN}
        \begin{cases}
            dX(t)= F(t,X(t)) dt + G(t,X(t)) dW(t) 
            \\ X(t_0)=X_0 \in \mathcal{H},
        \end{cases}
    \end{align*}
defined on a filtered probability space $(\Omega,\mathcal{F},\{\mathcal{F}_t\}_{t\geq t_0},\mathbb{P})$, where $W$ denotes a $\mathcal{H}$-valued Brownian motion, $F: [t_0, +\infty) \times\mathcal{H} \to \mathcal{H}$ and $G: [t_0, +\infty) \times\mathcal{H} \to {\mathcal L}(\mathcal{H}, \mathcal{H})$. We equip ${\mathcal L}(\mathcal{H}, \mathcal{H})$ with the Hilbert-Schmidt inner product and corresponding norm $\|\cdot\|_{HS}$.

In the following definition, we introduce the notion of a solution that we consider for \eqref{SDE-GEN}.

\begin{defi}\label{def:strong-solution}(\cite[Definition 3.1]{gawarecki-mandrekar})
    A stochastic process $X : \Omega \times [t_0, +\infty) \to \mathcal{H} $  is a strong solution of~\eqref{SDE-GEN} if 
    \begin{enumerate}[(i)]
        \item $X\in C([t_0,+\infty),\mathcal{H})$ almost surely;
        \item for every $t \geq t_0$ it holds
        \begin{align*}
            \int_{t_0}^t \big(\|F(s,X(s))\| + \|G(s,X(s))\|_{HS}^2\big) ds < +\infty \quad \mbox{almost surely};
        \end{align*}
        \item for every $t \geq t_0$ it holds
        \begin{align*}
            X(t) = X(t_0) + \int_{t_0}^t F(s,X(s)) ds + \int_{t_0}^t G(s,X(s))dW(s) \quad \mbox{almost surely}.
        \end{align*}
    \end{enumerate}
\end{defi}

Furthermore, we introduce a class of stochastic processes characterized by a key property governing the expectation of their norm raised to a given power.

 \begin{defi}\label{def21}
(i) A stochastic process $X:\Omega \times [t_0, +\infty) \rightarrow \mathcal{H}$ is called progressively measurable if for every $T \geq t_0$ the mapping
        \begin{align*}
            &\Omega\times [t_0,T] \rightarrow \mathcal{H}, \quad (\omega,s) \mapsto X(\omega,s),
        \end{align*}
        is $\mathcal{F}_T \otimes\mathcal{B}([t_0,T])$-measurable, where $\otimes$ denotes the product $\sigma$-algebra and $\mathcal{B}$ is the Borel $\sigma$-algebra. Further, $X$ is called adapted if $X(\cdot,T)$ is $\mathcal{F}_T$-measurable for every $T\geq t_0$.
        \item For $T \geq t_0$, we define the quotient space as
        \begin{align*}
            S_\mathcal{H}^0[t_0,T]:=\left\{X:\Omega\times[t_0,T]\rightarrow\mathcal{H}: \text{$X$ is a progressively measurable continuous stochastic process}\right\}\Big/\mathcal{R},
        \end{align*}
     and set $S_\mathcal{H}^0:=\bigcap_{T\geq t_0} S_\mathcal{H}^0[t_0,T]$. 
     
(ii) For $\nu>0$ and $T \geq t_0$, we define $S_\mathcal{H}^\nu [t_0,T]$ as the following subset of stochastic processes in $S_\mathcal{H}^0 [t_0,T]$
        \begin{align*}
            S_\mathcal{H}^\nu [t_0,T] := \left\{X\in S_\mathcal{H}^0 [t_0,T]: \quad \mathbb{E}\left(\sup_{t\in[t_0,T]} \|X(t)\|^\nu\right) < +\infty \right\}.
        \end{align*}
        Finally, we set $S_\mathcal{H}^\nu := \bigcap_{T\geq t_0} S_\mathcal{H}^\nu[t_0,T]$.
\end{defi}
The theorem below ensures the existence and uniqueness of solutions to \eqref{SDE-GEN} within a fixed compact interval.

\begin{thm}\label{thm:existence-uniqueness-solutions}
Consider the general stochastic differential equation \eqref{SDE-GEN} and let $T \geq t_0$. Assume that  $F$ and $G$ are jointly measurable, i.e. $F:[t_0,T]\times \mathcal{H}\rightarrow \mathcal{H}$ and $G:[t_0,T]\times \mathcal{H}\rightarrow \mathcal{L}(\mathcal{H},\mathcal{H})$ are  $\mathcal{B}([t_0,T])\otimes\mathcal{B}(\mathcal{H})$-measurable, and also measurable with respect to the product $\sigma$-field $\mathcal{B}([t_0,t])\otimes \mathcal{B}(\mathcal{H})$ on $[t_0,t]\times\mathcal{H}$ for every $t_0\leq t\leq T$.

\begin{enumerate}[(i)]
\item\label{thm:item:existence-uniqueness-strong-solution} (\cite[Theorem 3.3]{gawarecki-mandrekar}) Assume that the following conditions hold:
    \begin{enumerate}[(1)]
        \item There exists a constant $\ell$ such that for every $X \in C([t_0,T],\mathcal{H})$ almost surely it holds
        \begin{align*}
            \|F(t,X(\omega,t))\| + \|G(t,X(\omega,t))\|_{HS} \leq \ell \left(1+\sup_{t_0\leq s\leq T} \|X(\omega, s)\| \right) \ \mbox{for every} \ t_0 \leq t \leq T \ \mbox{and} \ \omega\in\Omega.
        \end{align*}
        \item For every $X,Y \in C([t_0,T],\mathcal{H}), t_0 \leq t \leq T$ and $\omega\in\Omega$, there exists $K > 0$ such that
        \begin{align*}
            \|F(t,X(\omega,t)) - F(t,Y(\omega,t))\| + \|G(t,X(\omega,t))-G(t,Y(\omega,t))\|_{HS} \leq K\sup_{t_0\leq s\leq T} \|X(\omega, s) - Y(\omega, s)\|.
        \end{align*}
    \end{enumerate}
 Then, ~\eqref{SDE-GEN} has a unique strong solution on $[t_0,T]$.
    \item(\cite[Theorem A.7]{soto-fadili-attouch}, \cite[Theorem 5.2.1]{oksendal})\label{thm:item:help-existence-uniqueness} Additionally, assume that there exists $C>0$ such that
    \begin{align*}
        \|F(t,x)-F(t,y)\| + \|G(t,x)-G(t,y)\|_{HS} \leq C \|x-y\| \quad \mbox{for every} \ x,y\in\mathcal{H} \ \mbox{and} \ t_0 \leq t \leq T.
    \end{align*}
 Then, the unique solution of the stochastic differential equation~\eqref{SDE-GEN} on $[t_0, T]$ lies in $S_{\mathcal{H}}^{\nu}[t_0,T]$ for every $\nu\geq 2$.
\end{enumerate}
\end{thm}

Using standard extension arguments, we then obtain a unique solution of the system \eqref{SDE-GEN} on $[t_0, +\infty)$ which also belongs to $S_{\mathcal{H}}^{\nu}$ for every $\nu \geq 2$.

Next, we recall the It\^o formula.

\begin{prop}\label{prop:ito-formula}\cite[Theorem 2.9]{gawarecki-mandrekar}
Let $T \geq t_0$ be fixed. Let $X$ be a stochastic process given by
\begin{align*}
    X(t) = X_0 + \int_{t_0}^{t} F(s) ds + \int_{t_0}^{t} G(s) dW(s) \quad \forall t\in[t_0,T],
\end{align*}
defined, as before, on a filtered probability space $(\Omega,\mathcal{F},\{\mathcal{F}_t\}_{t_0 \leq t \leq T},\mathbb{P})$, where $W$ denotes a $\mathcal{H}$-valued Brownian motion, $F:[t_0,T]\rightarrow \mathcal{H}$ and $G:[t_0,T]\rightarrow \mathcal{L}(\mathcal{H},\mathcal{H})$. Moreover, we require that $X_0$ is an $\mathcal{F}_{t_0}$-measurable $\mathcal{H}$-valued random variable, $F(\cdot)$ is adapted such that
\begin{align*}
    \int_{t_0}^{T} \|F(t)\| dt < +\infty \ \mbox{a.s.},
\end{align*}
and $G(\cdot)$ is adapted such that
\begin{align*}
    \int_{t_0}^{T} \|G(t)\|^2_{HS} dt < +\infty \ \mbox{a.s.}
\end{align*}
Let $\phi:[t_0, +\infty) \times\mathcal{H} \rightarrow \R$ be such that $\phi(\cdot,x)\in\text{C}^1([t_0, +\infty))$ for every $x\in\mathcal{H}$ and $\phi(t,\cdot)\in\text{C}^2(\mathcal{H})$ for every $t\geq t_0$. Then, 
    \begin{align*}
        \Tilde{X}(t):=\phi(t,X(t))
    \end{align*}
is an It\^{o} process such that
    \begin{align*}
        d\Tilde{X}(t) = \frac{d}{dt} \phi(t,X(t)) dt &+ \left\langle \nabla_x \phi(t,X(t)), F(t)\right\rangle dt + \langle \nabla_x\phi(t,X(t)),G(t) dW(t)\rangle
        \\&+ \frac{1}{2} \trace \left(\nabla_{xx}^2\phi(t,X(t)) G(t)G(t)^*\right) dt \quad \forall t \geq t_0.
    \end{align*}
    Here, $G(t)^*$ denotes the adjoint operator of $G(t)$ and $\trace(\cdot)$ the trace of an operator $A \in {\cal L}(\mathcal{H}, \mathcal{H})$, defined by
    \begin{align*}
        \trace(A) = \sum_{k=1}^\infty \langle Ae_k, e_k\rangle,
    \end{align*}
    where $\{e_k\}_{k\in\N}$ denotes an orthonormal basis of the Hilbert space $\mathcal{H}$.
    Further, if, for every $t_0 \leq t \leq T$,
    \begin{align*}
        \mathbb{E}\left(\int_{t_0}^t \|G(s)^*\nabla_x\phi(s,X(s))\|^2 ds\right)<+\infty,
    \end{align*}
    then $t \mapsto \int_{t_0}^t\left\langle G(s)^*\nabla_x\phi(s,X(s)), dW(s) \right\rangle$ is, for every $t_0 \leq t \leq T$, a square integrable continuous martingale with expected value $0$.
\end{prop}
Recall the definition of $x_{\varepsilon(t)}$: Let $F(t,x) = \nabla f(x) + \varepsilon(t)x$. Then, the implicit function theorem yields the existence of a continuously differentiable $x_{\varepsilon(t)}$ such that $F(t,x_{\varepsilon(t)}) = 0$ and $F(t,x) = 0 \Leftrightarrow x = x_{\varepsilon(t)}$. Our main result will use the lemma below for the estimation of $x_{\varepsilon(t)}$.
\begin{lem}(\cite[Lemma 2]{attouch-balhag-chbani-riahi})\label{lem:estimate-x-epsilon}
    As in the main body of the paper, let
\begin{align*}
    \varphi_t(x) := f(x) + \frac{\varepsilon(t)}{2}\|x\|^2.
\end{align*}
Then, the following statements are true:
    \begin{enumerate}[(i)]
        \item For each $t\geq t_0$, $\frac{d}{dt}(\varphi_t(x_{\varepsilon(t)})) = \frac{1}{2}\dot{\varepsilon}(t)\|x_{\varepsilon(t)}\|^2$.
        \item The function $t\mapsto x_{\varepsilon(t)}$ is Lipschitz continuous on the compact intervals of $(t_0,+\infty)$, hence differentiable almost everywhere, and the following inequality holds: For almost every $t\geq t_0$
        \begin{align*}
            \left\|\frac{d}{dt}x_{\varepsilon(t)}\right\|^2\leq -\frac{\dot{\varepsilon}(t)}{\varepsilon(t)}\left\langle \frac{d}{dt}x_{\varepsilon(t)}, x_{\varepsilon(t)}\right\rangle.
        \end{align*}
        Therefore, for almost every $t\geq t_0$,
        \begin{align*}
            \left\|\frac{d}{dt}x_{\varepsilon(t)}\right\| \leq -\frac{\dot{\varepsilon}(t)}{\varepsilon(t)}\|x_{\varepsilon(t)}\|.
        \end{align*}
    \end{enumerate}
\end{lem}
Another theorem that will be instrumental for the convergence analysis is detailed below.
\begin{thm}(\cite[Theorem 3.9]{mao})\label{thm:A.9-in-paper}
    Let $\{A(t)\}_{t\geq 0}$ and $\{U(t)\}_{t\geq 0}$ be two continuous adapted increasing processes with $A(0)=U(0)=0$ a.s. Let $\{N(t)\}_{t\geq 0}$ be a real-valued continuous local martingale with $N(0)=0$ a.s. Let $\xi$ be a nonnegative $\mathcal{F}_0$-measurable random variable. Define
    \begin{align*}
        X(t):=\xi + A(t) - U(t) + N(t) \quad \text{for} \quad \forall t\geq 0.
    \end{align*}
    If $X(t)$ is nonnegative and $\lim_{t\rightarrow +\infty}{A(t)} < +\infty$ a.s., then a.s. $\lim_{t\rightarrow +\infty}{X(t)}$ exists and is finite, as well as $\lim_{t\rightarrow +\infty}{U(t)} < +\infty$.
\end{thm}
The Gronwall-Bellman Lemma will also be of use to us and is recalled here for the reader's convenience.
\begin{lem}\label{lem:gronwall-bellman}
    Let $m:[\delta,T] \rightarrow [0,+\infty)$ be integrable, and let $c\geq 0$. Suppose $w:[\delta,T]\rightarrow \R$ ia continuous and
    \begin{align*}
        \frac{1}{2}w(t)^2 \leq \frac{1}{2}c^2 + \int_{\delta}^{t} m(s)w(s)ds
    \end{align*}
    for all $t\in [\delta, T]$. Then, $|w(t)| \leq c + \int_{\delta}^{t} m(s)ds$ for all $t\in[\delta,T]$.
\end{lem}

\printbibliography[]
\end{document}